\documentclass[11pt]{amsart}

\usepackage[margin=1in]{geometry}
\usepackage{xypic}

\usepackage{amsfonts,amssymb,amsthm,amsmath,euscript}

\usepackage[colorlinks=true,citecolor=blue,linkcolor=red]{hyperref}
\usepackage[alphabetic,msc-links]{amsrefs}

\usepackage{graphicx,setspace,amscd,float,enumitem} 

\makeatletter
\newtheorem*{rep@theorem}{\rep@title}
\newcommand{\newreptheorem}[2]{\newtheorem*{rep@#1}{\rep@title}\newenvironment{rep#1}[1]{\def\rep@title{#2 \ref*{##1}}\begin{rep@#1}}{\end{rep@#1}}}
\makeatother

\newtheorem{thm}{Theorem}[section]
\newtheorem{lem}[thm]{Lemma}

\newtheorem{prop}[thm]{Proposition}
\newtheorem{cor}[thm]{Corollary}
\theoremstyle{definition}
\newtheorem{df}[thm]{Definition}
\newtheorem{rem}[thm]{Remark}
\newtheorem{ex}[thm]{Example}
\newtheorem{conv}[thm]{Convention}

\newtheorem{theor}{Theorem} 
\newreptheorem{theor}{Theorem}

\newtheorem{corol}[theor]{Corollary} \newreptheorem{corol}{Corollary}

\newtheorem{apptheor}[thm]{Theorem}
\newreptheorem{apptheor}{Theorem}
\newtheorem{appcorol}[thm]{Corollary}
\newreptheorem{appcorol}{Corollary}

\newcommand{\Z}{\mathbb Z}
\newcommand{\Q}{\mathbb Q}

\newcommand{\PP}{\mathbb P}
\newcommand{\tJ}{\widetilde{J}}
\newcommand{\tj}{\widetilde{j}}

\newcommand\inter{\mathop{\rm int}}

\newcommand\wasM{N}
\newcommand\twasM {\widetilde{N}}
\newcommand\wasN{M}
\newcommand\tS{\widetilde{S}}
\newcommand\genus{\mathop{\rm genus}}
\def\epsilon{\varepsilon}
\def\phi{\varphi}

\newcommand{\Out}{\mbox{Out}}
\newcommand{\Aut}{\mbox{Aut}}

\newcommand{\rank}{\mathrm{rank}}
\newcommand{\rr}{\mbox{\rm rr}}

\begin{document}

\title[On two-generator subgroups of mapping torus groups]{On two-generator subgroups of mapping torus groups}

\author{Naomi Andrew}
\address{Mathematical Institute, Andrew Wiles Building, Observatory Quarter, University of Oxford, Oxford OX2 6GG, United Kingdom} \email{\tt Naomi.Andrew@maths.ox.ac.uk}

\author{Edgar A. Bering IV}
\address{Department of Mathematics and Statistics, San Jos\'{e} State University, One Washington Square, San Jose, CA 95192, U.S.A.} \email{\tt edgar.bering@sjsu.edu} 

\author{Ilya Kapovich}
\address{Department of Mathematics and Statistics, Hunter College of CUNY,
  695 Park Avenue, New York, NY 10065, U.S.A.
  \newline \indent\tt http://math.hunter.cuny.edu/ilyakapo/, } \email{\tt ik535@hunter.cuny.edu}

\author{Stefano Vidussi}
\address{Department of Mathematics, University of California Riverside, 900 University Ave., Riverside, CA 92501, U.S.A.} \email{\tt svidussi@ucr.edu}

\makeatletter
\let\@wraptoccontribs\wraptoccontribs
\makeatother
\contrib[With an appendix by]{Peter Shalen}
\address{Department of Mathematics, Statistics, and Computer Science (M/C 249), University of
Illinois at Chicago, 851 S. Morgan St., Chicago, IL 60607} \email{\tt petershalen@gmail.com}

\begin{abstract}
We prove that if $G_\phi=\langle F, t| t x t^{-1} =\phi(x), x\in F\rangle$ is
the mapping torus group of an injective endomorphism $\phi: F\to F$ of a free
group $F$ (of possibly infinite rank), then every two-generator subgroup $H$ of
$G_\phi$ is either free or a (finitary) sub-mapping torus. As an application we show that if $\phi\in \Out(F_r)$ is a fully irreducible atoroidal automorphism then every two-generator subgroup of $G_\phi$ is either free or has finite index in $G_\phi$.
\end{abstract}

\makeatletter
\@namedef{subjclassname@2020}{\textup{2020} Mathematics Subject Classification}
\makeatother

\subjclass[2020]{Primary 20F65, Secondary  20F05, 57M}

\date{}
\maketitle

\tableofcontents

\section{Introduction}

Characterizing subgroups with a given number of generators in a group
from some natural class is usually too difficult. However, two-generator subgroups often
have more restricted algebraic structure, particularly in various geometric
contexts. A well-known result of Delzant~\cite{D91}*{Th\'eor\`eme~2.1} shows that in a
torsion-free word-hyperbolic group $G$ there are only finitely many conjugacy
classes of non-elementary freely indecomposable two-generator subgroups. In the setting of right-angled Artin groups, a classical result of Baudisch states that any two-generator subgroup is either free-abelian or free~\cite{B81}*{Corollary 1.3}. Antol\'{i}n and Minasyan generalize this to a structure theorem for two-generator subgroups of graph products of torsion-free groups in a product and subgroup closed family~\cite{AM15}*{Corollary 1.5}.
An
even older result of Jaco and Shalen from 1979~\cite{JS79}*{Theorem~VI.4.1}
proves that if $M$ is an orientable\footnote{All 3-manifolds considered by Jaco
and Shalen are assumed to be orientable~\cite{JS79}*{convention p. 2}. See the
appendix to this paper for the extension to the non-orientable setting.}
atoroidal Haken
3-manifold and $H$ is a two-generator
subgroup of $G=\pi_1(M)$ then either $H$ is free (of rank $\le 2$) or $H$ is
free abelian (again of rank $\le 2$) or $H$ has finite index in
$G$.
(A result similar to Jaco and Shalen's was proved independently and
contemporaneously by Thomas Tucker in \cite{Tuc77}.)
This
theorem applies, for example, to the case where $M$ is a complete finite-volume
hyperbolic Haken 3-manifold without boundary.
Note that in this case if $M$ is closed
then the Jaco--Shalen result implies that every two-generator subgroup of
$G=\pi_1(M)$ is either free or has finite index in $G$.

In this paper we obtain an analog of the Jaco--Shalen result for two-generator subgroups of mapping tori of free group endomorphisms.

Let $F$ be a free group (of possibly infinite rank), let $\phi:F\to F$ be an injective endomorphism of $F$. We call the group
\[
G_\phi=\langle F, t|\, t x t^{-1} =\phi(x), x\in F\rangle
\]
the \emph{mapping torus group of $\phi$.} Here $\phi$ may be an automorphism of $F$ in which case $G_\phi$ is a free-by-cyclic group $G_\phi=F\rtimes_\phi \mathbb Z$.
A finitely generated subgroup $H\le
 G_\phi$ is called a \emph{finitary sub-mapping torus
of $G_\phi$} or just \emph{sub-mapping torus of $G_\phi$} if there exist 
$u\in F$, $m\ge 1$ and a nontrivial
finitely generated subgroup $V\le F$ such that with $s=ut^m$, the group $H=\langle s, V\rangle$ and
$sVs^{-1}\le V$. For brevity we will typically use the term ``sub-mapping torus
of $G_\phi$", but it is important to remember that in this definition $V$ is required to be
finitely generated.

 As we see in Lemma~\ref{lem:sub} below, if $H$ is a sub-mapping torus of $G_\phi$ as above, then $H$ has the presentation
\[
H=\langle s, V| sys^{-1}=\psi(y), y\in V\rangle
\]
where $\psi:V\to V$ is an injective endomorphism of the free group $V$ given by $\psi(y)=u\phi^m(y)u^{-1}$ for all $y\in V$. Thus $H=G_\psi$ is the mapping torus group of the injective endomorphism $\psi:V\to V$ of a finite rank free group $V$. This fact explains the term ``sub-mapping torus."

\begin{theor}\label{thm:A}

Let $F$ be a free group, let $\phi:F\to F$ be an injective endomorphism of $F$ and let $G_\phi$
be the mapping torus group of $\phi$.

Then for every two-generator subgroup $H\le G_\phi$ either $H$ is free or $H$
	is conjugate to a finitary sub-mapping torus subgroup of $G_\phi$. 
\end{theor}

\begin{rem}
	If $\phi:F\to F$ is an injective endomorphism of a free group $F$ such that $G_\phi$ is finitely generated, and we include the additional assumption that $H$ is {\em normal} in $G_\phi$, then (even without requiring the finite rank of $H$ be two),
	it is well-known that we can get a stronger conclusion. In fact, $G_\phi$ has cohomological dimension at most two \cite{Bie81}*{Proposition 6.2}, and coherence of $G_\phi$ \cite{FH99} entails that $H$ is
	finitely presented. Therefore by a theorem of
	Bieri $H$ is either free or finite
	index~\cite{Bie76}*{Theorem B}.  
\end{rem}

Note that standard examples of $\mathbb Z\times \mathbb Z$ and, more generally, Baumslag-Solitar subgroups of $G_\phi$ are covered by Theorem~\ref{thm:A} and appear as sub-mapping tori. These subgroups arise from the situations where there exist some $1\ne x\in F$, $p\ne 0$, $u\in F$, $m\ge 1$ such that $sxs^{-1}=x^p$, where $s=ut^m$. There are also more complicated two-generated sub-mapping torus subgroups that may arise in $G_\phi$, as demonstrated by the following example.

\begin{ex}
Let $F=F(a,b,c,d,e)$ and consider an automorphism $\phi\in \Aut(F)$ given by
	\[\phi(a)=b,\, \phi(b)=c, \,\phi(c)=ab^2c^2, \,\phi(d)=dea,\,
	\phi(e)=ede.\]
Put $G_\phi=\langle F, t| txt^{-1}=\phi(x), x\in F\rangle$. Note that $K=\langle a,b,c\rangle$ is a proper free factor of $F$ with $\phi(K)=K$.
Consider the subgroup $H=\langle t,a\rangle\le G_\phi$. Then $H=G_\psi$ is the mapping torus of the automorphism $\psi$ of $K$ where $\psi=\phi|_K$, and $H$ is a sub-mapping torus of $G_\phi$.
\end{ex}

The proof of Theorem~\ref{thm:A} relies on the proof of coherence of mapping tori of free group endomorphisms given by Feighn and Handel~\cite{FH99}. In their proof, Feighn and Handel introduced the notion of an \emph{invariant graph pair} (see Section~\ref{S:pairs}) for representing finitely generated subgroups of such mapping tori. They defined the notion of \emph{relative complexity} for invariant graph pairs and proved that there exist \emph{tight} (or \emph{folded}) invariant graph pairs of minimal complexity and that one can read-off the presentation of a subgroup from having such a pair. In general, finding a tight invariant subgroup pair of minimal complexity is a non-constructive step in Feighn and Handel's proof of coherence. However, for two-generator subgroups we are able to overcome this difficulty and analyze precisely what happens algebraically because the two-generator assumption implies that the minimal relative complexity for subgroups under consideration is either $0$ or $1$.

If $\phi\in\Aut(F)$, then the isomorphism class of $G_\phi$ depends only on the outer automorphism class of $\phi$ in $\Out(F)$. For that reason, if $\phi\in \Out(F)$, by a slight abuse of notation we will denote by $G_\phi$ the group $G_\psi$, where $\psi\in \Aut(F)$ is any representative of the outer automorphism class $\phi$. 

In the case where $\phi \in\Aut(F)$ then every subgroup of
	$G_\phi$ conjugate to a sub-mapping torus subgroup of $G_\phi$ is
	itself a sub-mapping torus subgroup of $G_\phi$. Therefore in the case
	where $\phi$ is an automorphism of $F$, the ``conjugate to" phrase in
	the conclusion of Theorem~\ref{thm:A} can be dropped. Further, the
	structure and geometry of a sub-mapping torus can be determined.

\begin{corol}\label{cor:distort}
	Suppose $G_\phi$ is a free-by-cyclic group $F_r \rtimes_\phi \Z$. Then
	a two-generator subgroup of $G_\phi$ is either free of rank 2, or
	an undistorted (f.g. free)-by-cyclic sub-mapping torus subgroup of $G_\phi$.
\end{corol}

Recall that for an integer $r\ge 2$, an element $\phi\in\Aut(F_r)$ (or the
corresponding outer automorphism class in $\Out(F_r)$) is called \emph{fully
irreducible} if there do not exist a power $p\ge 1$ and a proper free factor
$1\ne U\lneq F_r$ of $F_r$ such that $\phi^p([U])=[U]$, where $[U]$ is the
conjugacy class of $U$ in $F_r$.  Also an element $\phi\in\Aut(F_r)$ (or the
corresponding outer automorphism class in $\Out(F_r)$) is called
\emph{atoroidal} if there do not exist a power $p\ge 1$ and an element $1\ne u\in F_r$
such that $\phi^p([u])=[u]$, where again $[u]$ is the conjugacy class of $u$ in
$F_r$. The notion of being fully irreducible provides the main $\Out(F_r)$
analog of being a pseudo-Anosov mapping class group element~\cite{CH12}. By a result of Brinkmann~\cite{B00}*{Theorem~1.2},
for $\phi\in\Out(F_r)$ the group $G_\phi$ is word-hyperbolic if and only if
$\phi$ is atoroidal. It is known that for most reasonable random
walks on $\Out(F_r)$, where $r\ge 2$, a ``random" element of $\Out(F_r)$ is
fully irreducible, and for $r\ge 3$ it is also atoroidal by~\cite{Rivin}*{Theorem 4.1} or~\cite{Si18}*{Theorem~1.4, Proposition~3.9}.

As an application of Theorem~\ref{thm:A} we obtain:

\begin{theor}\label{thm:B}

Let $F_r$ be a free group of finite rank $r\ge 2$, let $\phi\in \Out(F_r)$ be a fully irreducible outer automorphism and let $G_\phi$
be the mapping torus group of $\phi$.

Then for every two-generator subgroup $H\le G_\phi$ either $H$ is free, free abelian, a Klein bottle group, or $H$ has finite index in $G_\phi$. 
\end{theor}
Here by the Klein bottle group we mean the Baumslag-Solitar group $BS(1,-1)=\langle a,b| b ab^{-1} =a^{-1}\rangle$.

\begin{rem}
In Theorem~\ref{thm:B} (and similar statements below), $H  \le G_\phi$ can be finite index only if $b_1(G_\phi) \leq 2$, as by the existence of transfer, the first Betti number is non-decreasing over finite index subgroups (see Lemma~\ref{lem:tra} below).
\end{rem}
\begin{rem}
An obvious way to get 2-generated free subgroups in a free-by-cyclic group
	$G_\phi$ such as the one in  Theorem \ref{thm:B} is as subgroups of
	fibers of algebraic fibrations of $G_\phi$. In fact, one can see that
	not all 2-generated free subgroups of $G_{\phi}$ arise this way: For
	instance, one can construct a group $G_\phi$ as in Theorem \ref{thm:B}
	(for $r \geq 3$, so that $G_{\phi}$ is word-hyperbolic) such that
	$b_1(G_{\phi}) = 1$ in which case the subgroup $H = \langle
	x^n,t^n\rangle$ is free, for sufficiently large $n$, but is not
	contained in the unique fiber of $G_{\phi}$. On the other hand, it is not clear whether one could find a finite-index subgroup of $G_{\phi}$ admitting a fibration whose fiber contains $H$.
\end{rem}

If $\phi\in\Out(F_r)$ is fully irreducible but not atoroidal, then $G_\phi$ is a 3-manifold group. In this case the conclusion of Theorem~\ref{thm:B} might be derived using the methods of Jaco and Shalen mentioned above. As we explain below, the non-orientable case presents some extra complication to the 3-manifold approach, requiring additional work.

In fact, using properties of the pseudo-Anosov homeomorphism we generalize one case of the Jaco--Shalen result by removing the orientability assumption:

\begin{corol}\label{cor:3manifold}
	Let $G$ be the fundamental group of a fibered cusped hyperbolic 3-manifold. Then every two-generator subgroup $H \le G$ is either free, free abelian, a Klein bottle group, or $H$ has finite index in $G$. 
\end{corol}

One might expect an easier alternate argument for the proof of Theorem~\ref{thm:B} in the case where $\phi\in\Out(F_r)$ is fully irreducible but not atoroidal, using 3-manifold methods.
Since $\phi\in\Out(F_r)$ is fully irreducible but not atoroidal, a result of Bestvina and Handel~\cite{BH95} implies that $\phi$ is induced by a pseudo-Anosov homeomorphism $f$ of a compact connected surface $\Sigma$ with a single boundary component. Then the mapping torus $M$ of $f$ is a (truncated) finite volume complete hyperbolic 3-manifold with a single cusp which has $G_\phi=\pi_1(M)$.  Note that $M$ is orientable if and only if $\Sigma$ is orientable and $f$ is orientation preserving.
With this set-up one might hope to use the result of Jaco and Shalen \cite{JS79}*{Theorem~VI.4.1} that motivates this paper to obtain the same conclusion about two-generator subgroups of $G_\phi$: that every such subgroup is either free or free abelian or the Klein bottle group or has finite index in $G_\phi$. However, it is not straightforward to apply that result to the non-orientable case, and some further techniques from Jaco and Shalen's work are needed to handle non-orientable fibers. (A more recent paper of Boileau and Weidmann~\cite{BW05} about 3-manifolds with two-generated fundamental group also only considers the orientable case.)  The same issue arises in trying to apply the results of Jaco and Shalen \cite{JS79} to obtain a 3-manifold proof of Corollary~\ref{cor:3manifold}.

Nevertheless, in the appendix to the present paper Shalen remedies
this situation and generalizes the relevant results 
to include $\PP^2$-irreducible  3-manifolds that may be non-orientable 
(and may fail to be Haken manifolds, i.e. to be ``sufficiently
large'').
In particular, Corollary~\ref{appendix-prop} in the appendix implies the conclusion of Corollary~\ref{cor:3manifold}  (and thus also of Theorem~\ref{thm:B} in the fully irreducible but non-atoroidal case).

\begin{repappcorol}{appendix-prop}[P. Shalen]
Let $\wasM $ be a compact, $\PP^2$-irreducible  $3$-manifold such that
every rank-$2$ free
abelian subgroup of $\pi_1(\wasM )$ is peripheral. Let $H$ be a subgroup
of $\pi_1(\wasM )$ that has rank at most $2$. Then $H$ is either a
free group, a free abelian group, a Klein bottle group, or a
finite-index subgroup of $\pi_1(\wasM )$.
\end{repappcorol}
 
This result is in turn deduced from a more general result showing that for any fixed integer $k\ge 2$, under some additional algebraic assumptions about the fundamental group of a 3-manifold $N$, every $k$-generated subgroup of $G=\pi_1(M)$ is isomorphic to a free product of some free group, some number of copies of $\mathbb Z^2$ and some number of copies of the Klein bottle group.

\begin{repapptheor}{k-gen}[P. Shalen]
Let $\wasM $ be a compact, $\PP^2$-irreducible  $3$-manifold, and let
$k\ge2$ be an integer. Suppose that 
every rank-$2$ free
abelian subgroup of $\pi_1(\wasM )$ is peripheral. Suppose also that
either 

(a) $\wasM$ is orientable and $\pi_1(\wasM)$ has no subgroup isomorphic to the
fundamental group of a closed, orientable surface $S$ such that
$1<\genus S<k$, or 

(b) $\pi_1(\wasM)$ has no subgroup isomorphic to the
fundamental group of a closed surface $S$ such that $\chi(S)$ is even and
$2-2k<\chi( S)<0$. 

Let $H$ be a subgroup
of $\pi_1(\wasM )$ that has rank at most $k$. Then either $H$ has
finite index in $\pi_1(\wasM)$, or $H$ is a free product of finitely
many subgroups, each of which is  a
 free abelian group of rank at most $2$ or a Klein bottle group.
\end{repapptheor}

It would be interesting to understand if there exist any reasonable analogues of Theorem~\ref{k-gen} in the context of free-by-cyclic groups or, more generally, mapping tori of free group endomorphisms.

Returning to outer automorphisms, when $\phi\in\Out(F_r)$ is fully irreducible and atoroidal (which necessarily implies that $r\ge 3$), $G_\phi$ is a word-hyperbolic group~\cite{B00}*{Theorem~1.2} and hence $G_\phi$ contains no $\mathbb Z\times \mathbb Z$ subgroups (and therefore $G_\phi$ cannot contain a Klein bottle subgroup either).  Therefore in this case Theorem~\ref{thm:B} implies a sharper corollary. 

\begin{corol}\label{cor:A}
	Let $F_r$ be a free group of finite rank $r\ge 3$, let $\phi\in\Out(F_r)$ be a fully irreducible atoroidal outer automorphism. 
	
	Then every two-generator subgroup $H\le G_\phi$ of the mapping torus group is either free or has finite index in $G_\phi$. 
\end{corol}

It turns out that in the above applications of Theorem~\ref{thm:A} one can remove the ``or has finite index in $G$" alternative in the conclusion of these results by replacing $G$ by a suitable subgroup $G_0$ of finite index at the outset:

\begin{corol}\label{cor:G0} 
For each of the following groups $G$ there exists a finite index subgroup $G_0$ such that every two-generator subgroup of $G_0$ is either free, free abelian, or the Klein bottle group.
	\begin{enumerate}
		\item If $\phi\in\Out(F_r)$ is fully irreducible, $r\ge 2$ and $G = G_\phi$. If $\phi$ is additionally atoroidal, every two-generator subgroup of $G_0$ is free.
		\item If $G$ is the fundamental group of a fibered cusped hyperbolic 3--manifold.
	\end{enumerate}
\end{corol}

Another application of Theorem~\ref{thm:A} (where cyclic extensions of infinite rank free groups play a role) concerns one-relator groups with torsion. Pride~\cite{P77} proved that if $G=\langle x_1,\dots, x_r| w^n=1\rangle$, where $w$ is a nontrivial freely and cyclically reduced word and where $n\ge 2$, then every two-generator subgroup of $G$ is either a free product of two cyclic groups or itself is a one-relator group with torsion. Since $G$ as above is known to be virtually torsion-free, it follows that $G$ has a subgroup of finite index $G_0$ such that every two-generator subgroup of $G_0$ is free. We recover a weaker version of this result, under the extra assumption that $n$ is big enough.

\begin{theor}\label{thm:C}
Let $G=\langle x_1,\dots, x_r| w^n=1\rangle$ where $r\ge 2$, $w\in F(x_1,\dots, x_r)$ is a nontrivial freely and cyclically reduced word and $n$ is an integer such that $n\ge |w| \ge 2$. Then $G$ contains a subgroup of finite index $G_0$ such that every two-generator subgroup of $G_0$ is free.
\end{theor}
Note that the subgroup $G_0 \leq G$ is free-by-cyclic as well. We also obtain a similar in spirit application for RFRS groups.

\begin{corol}\label{cor:RFRS}
	Suppose $G$ is a finitely generated RFRS group with $b_2^{(2)}(G)=0$ and $\mathrm{cd}_\mathbb{Q}(G)=2$. Then there is a finite index subgroup $G_0$ of $G$ such that every two generator subgroup $H$ of $G_0$ is free or the mapping torus of an endomorphism of a finitely generated free group.
\end{corol}

\subsection*{Acknowledgements}

We would like to thank the American Institute of Mathematics for the
hospitality during the October 2023 workshop ``Rigidity properties of
free-by-cyclic groups'' where this work started. We are also grateful to
Michael Kapovich for helpful discussions about 3-manifold groups, and to Chris
Leininger for discussion on pseudo-Anosov homeomorphisms and especially for
suggesting a specific proof of Lemma~\ref{lem:sub_pA}. Finally, we thank the
quick-opinion provider for suggesting we note Corollary~\ref{cor:distort} and
the anonymous referee for their careful reading and suggestions to improve the
exposition.

This work has received funding from the European Research Council (ERC) under the European Union's Horizon 2020 research and innovation programme (Grant agreement No. 850930).

I.~Kapovich acknowledges partial funding by the NSF grant DMS-1905641.

On behalf of all authors, the corresponding author, E. A. Bering, states that there is no conflict of interest.

\section{Labelled graphs and subgroups of free groups}

Let $F$ be a free group with a free basis $A$. We will utilize the standard
machinery of Stallings subgroup graphs and Stallings folds for working with
subgroups of $F=F(A)$. We direct the reader to standard references~\cites{St83,KM02} and only briefly recall some relevant notions here. An \emph{$A$-graph} is a graph (i.e. a 1-complex) $X$ where every oriented edge $e$ is given a label $\mu(e)\in A^{\pm 1}$ so that for every $e$ we have $\mu(e^{-1})=(\mu(e))^{-1}$. The \emph{$A$-rose} $R_A$ is an $A$-graph $R_A$ with a single vertex $v_0$ and one loop edge at $v_0$ labelled by $a\in A$ for each $a\in A$. For an $A$-graph $X$, the edge-labeling $\mu$ for $X$ defines a map $f_X:X\to R_A$ that send all vertices of $X$ to $v_0$ and maps each open $1$-cell of $X$ labelled by $a\in A$ homeomorphically, respecting orientation, to the open $1$-cell labelled by $a$ in $R_A$. An $A$-graph $X$ is \emph{tight} or \emph{folded} if the map $f_X: X\to R_A$ is an immersion; that is, if there do not exist two distinct oriented edges $e_1,e_2$ in $X$ with the same initial vertex $o(e_1)=o(e_2)$ and with $\mu(e_1)=\mu(e_2)$.

We naturally identify $F=F(A)=\pi_1(R_A,v_0)$.
If an $A$-graph $X$ has a base-vertex $\ast$, we denote by $X^\#$ the subgroup $$X^\#=(f_X)_\#(\pi_1(X,\ast))\le F=F(A).$$
Thus $X^\#$ consists of all the elements of $F=F(A)$ represented by labels of closed loops in $X$ at $\ast$.
Note that if $X$ is folded then $(f_X)_\#: \pi_1(X,\ast)\to X^\#$ is an isomorphism.

Suppose $X$ is an $A$-graph that is not tight and that $e_1,e_2$ are two distinct oriented edges in $X$ with $o(e_1)=o(e_2)$ and $\mu(e_1)=\mu(e_2)$. A \emph{Stallings fold} on the pair $e_1,e_2$ is a map $q:X\to X'$ that identifies $e_1,e_2$ into a single edge with label $\mu(e_1)=\mu(e_2)$, and identifies the terminal vertices $t(e_1)$, $t(e_2)$ if they were distinct in $X$.

\section{Standard subgroups of mapping torus groups and invariant graph pairs}\label{S:pairs}

\begin{conv}\label{conv:G}
For the remainder of this paper, unless specified otherwise we assume that $F$ is a free group, $\phi:F\to F$ is an injective endomorphism of $F$ and
\[
G_\phi=\langle F, t|\, t x t^{-1} =\phi(x), x\in F\rangle
\]
is the mapping torus group of $\phi$.
We also fix a free basis $A$ of $F$.

\end{conv}

\begin{df}\label{def:standard}
We say that a finitely generated subgroup $H\le G_\phi$ is \emph{standard} if $H=\langle t, V\rangle$ where $V\le F$ is a nontrivial finitely generated subgroup of $F$.
\end{df}

\begin{df}[Invariant graph pairs]
Let $H\le G_\phi$ be a standard finitely generated subgroup.
Let $Z$ be a finite connected $A$-graph with a base-vertex $\ast$ and let $X\subseteq Z$ be a connected subgraph containing $\ast$. We say that the pair $(Z,X)$ is an \emph{invariant graph pair for $H$} if the following conditions hold:
\begin{enumerate}
\item $H=\langle t, X^\#\rangle$;
\item $Z^\#=\langle X^\#, \phi(X^\#)\rangle$;
\item $Z^\#\le H\cap F$;
\end{enumerate}

\section{Relative complexity}

For an invariant graph pair $(Z,X)$, its \emph{relative rank} $\rr(Z,X)$ is defined as 
\[
\rr(Z,X)=\rank(\pi_1(Z))-\rank(\pi_1(X))=\chi(X)-\chi(Z).
\]
\end{df}

Let $(Z,X)$ be an invariant graph pair for $H$. Let $T\subseteq Z$ be a maximal subtree in $Z$ such that $T'=T\cap X$ is a maximal subtree in $X$. For each (oriented) edge $e$ of $Z-T$ labelled by a letter from the free basis $A$ of $F$ we get a loop $\gamma_e=[\ast, o(e)]_Te[t(e),\ast]_T$. The labels of these loops give us generating sets $\{c_1,\dots, c_k\}$ of $X^\#$ and $\{c_1,\dots, c_k, d_1,\dots, d_q\}$ of $Z^\#$. We call them the \emph{associated} generating sets of $X^\#$ and $Z^\#$ for $(Z,X)$. Note that $q=\rr(Z,X)$.

\begin{df}[Minimal relative complexity]
Let $H\le G_\phi$ be a standard finitely generated subgroup. Denote by $\rr_0(H)$ the minimum relative rank among all invariant graph pairs for $H$.
\end{df}

The following useful observation is an immediate consequence of the definitions:

\begin{lem}\label{lem:substandard}
Let $H\le G_\phi$ be a standard finitely generated subgroup. Suppose that $\rr_0(H)=0$. Then $H$ is a standard sub-mapping torus of $G_\phi$.
\end{lem}

\begin{proof}
Let $(Z,X)$ be an invariant graph pair for $H$ with $\rr(Z,X)=0$. Hence the inclusion of $X$ into $Z$ is a homotopy equivalence and we may make $Z$ smaller if needed and assume that $X=Z$.  Hence $Z^\#=\langle X^\#, \phi(X^\#)\rangle=X^\#$, and therefore $\phi(X^\#)\le X^\#$. Thus $H=\langle t,  X^\#\rangle$ is a standard sub-mapping torus of $G_\phi$, as claimed.
\end{proof}

Feighn and Handel define a \emph{tightening} process for reducing the relative rank of invariant graph pairs~\cite{FH99}.
Moves in this process are either ordinary Stallings folds or ``adding-of-loop" moves, where the latter can only occur after the so-called ``exceptional" folds. The relative rank is monotone non-increasing throughout the tightening process. The details of the tightening process are not important for our purposes, but we record here several key results about it.

\begin{prop}\label{prop:key}
Let $H\le G_\phi$ be a standard finitely generated subgroup.
\begin{enumerate}
	\item\label{prop:key-1} There exists a tight invariant graph pair $(Z,X)$ for $H$ with $\rr(Z,X)=\rr_0(H)$.
	\item\label{prop:key-2} Suppose that $(Z,X)$ is a tight invariant graph pair for $H$ with $\rr(Z,X)=\rr_0(H)$. Let $\{c_1,\dots, c_k\}$  and $\{c_1,\dots, c_k, d_1,\dots, d_q\}$ be associated generating sets of $X^\#$ and $Z^\#$ accordingly. Thus for $i=1,\dots, k$ there exists a word $w_i=w_i(c_1,\dots, c_k, d_1,\dots, d_q)$ such that $tc_it^{-1}=w_i$. 

Then the subgroup $H$ has the presentation
\[
H=\langle c_1,\dots, c_k, d_1,\dots, d_q, t| t c_it=w_i, i=1,\dots, k\rangle.
\] 

\end{enumerate}
\end{prop}

\begin{proof}
These statements are reformulations or immediate consequences from Feighn and Handel's paper~\cite{FH99}, this proof refers to their results.

	Part (1) follows from \cite{FH99}*{Proposition~5.4}, which states that relative rank is monotone non-increasing under tightening.

	Part (2) is a restatement of the proof of the Main Proposition in \cite{FH99}*{Section 6}.
\end{proof}

We also need the following classic 1939 result of Magnus~\cite{Mag39}*{Satz 1}.\footnote{An English presentation can be found in Magnus, Karrass, and Solitar \cite{MKS}*{p. 354, Corollary~5.14.2}. Magnus' result can also be proved using homological methods~\cite{Stamm67} or representation variety techniques~\cite{Li97}}

\begin{prop}\label{p:Magnus}
Let $m\ge n\ge 0$ be integers and let $G=\langle a_1,\dots, a_m| r_1,\dots
	r_n\rangle$ be a group given by a presentation with $m$ generators and $n$ defining relators. Suppose also that $G$ can be generated by $m-n$ elements. Then $G$ is free of rank $m-n$. 
\end{prop}

\begin{cor}\label{cor:rank2} If $H$ s a standard subgroup of $G_\phi$ of the form $H=\langle t,v\rangle$, where $v\in F$, and
	such that $\rr_0(H)=1$, then $H$ is free of rank 2.
\end{cor}

\begin{proof}
	By Proposition~\ref{prop:key}~(\ref{prop:key-1}) there is a tight invariant
	graph pair $(Z,X)$ for $H$ with $\rr(Z,X) = \rr_0(H) = 1$. The
	associated generating sets of $X^\#$ and $Z^\#$ have the form
	$\{c_1,\ldots, c_k \}$ and $\{c_1,\ldots, c_k, d_1\}$ where $k =
	\rank(\pi_1(X))$.
	Proposition~\ref{prop:key}~(\ref{prop:key-2}) then implies that there exists
	words $w_i = w_i(c_1,\ldots c_k,d_1)$ where $tc_it^{-1} = w_i$ and $H$ has the presentation
\[
H=\langle t, c_1,\dots, c_k, d_1| t c_it=w_i, i=1,\dots, k\rangle.
\]
	This presentation has $k+2$ generators and $k$ defining relators. By assumption $H=\langle t,v\rangle$ is two-generated. Therefore, by Proposition~\ref{p:Magnus}, $H$ is free of rank 2.
\end{proof}

\begin{prop}\label{prop:standard}
Let $1\ne v\in F$ and $H=\langle t,v\rangle \le G_\phi$. Then either $H$ is free of rank 2 or $H$ is a standard sub-mapping torus of $G_\phi$.
\end{prop}

\begin{proof}
Let $X'$ be a loop at $\ast$ labelled by the freely reduced form of $v$ and let $Z'$ be obtained from $X'$ by wedging it at $\ast$ with a loop labelled by the freely reduced form of $\phi(v)$ in $F$. Then $(Z',X')$ is an invariant graph pair for $H$ with $\rr(Z',X')=1$. Therefore $\rr_0(H)\le 1$.

If $\rr_0(H)=0$, then by Lemma~\ref{lem:substandard} $H$ is a standard sub-mapping torus of $G_\phi$.
If $\rr_0(H)=1$, then because $H=\langle t,v\rangle $ is a two-generated standard subgroup of $G_\phi$ by hypothesis, Corollary~\ref{cor:rank2} implies $H$ is free of rank 2.
\end{proof}

\section{Subgroups of mapping torus groups}

Recall that in Convention~\ref{conv:G} we have fixed a mapping torus group
$G_\phi$ for an injective endomorphism $\phi:F\to F$. First we record the following standard consequences of the definition of $G_\phi$:

\begin{lem}\label{lem:m} With $G_\phi$ and $\phi:F\to F$ as in
	Convention~\ref{conv:G},
\begin{enumerate}
\item For any $u\in F$ and $s=ut$ the group $G_\phi=\langle F, s\rangle$ has the presentation
\[
G_\phi=\langle F, s|\, s x s^{-1} =\xi(x), x\in F\rangle
\]
where $\xi:F\to F$ is an injective endomorphism given by $\xi(x)=u\phi(x)u^{-1}$ for all $x\in F$.
\item For any integer $m\ge 1$ the subgroup $\langle t^m, F\rangle$ has index $m$ in $G_\phi$ and with $s=t^m$ this subgroup has the presentation
\[
\langle F, s|\, s x s^{-1} =\phi^m(x), x\in F\rangle\cong G_{\phi^m}
\]
\item\label{lem:m-3} For any $u\in F$, $m\ge 1$ the subgroup $\langle ut^m, F\rangle=\langle t^m, F\rangle$ has index $m$ in $G_\phi$ and with $s=ut^m$ this subgroup has the presentation
\[
\langle F, s|\, s x s^{-1} =\psi(x), x\in F\rangle.
\] where $\psi:F\to F$ is an injective endomorphism given by $\psi(x)=u\phi^m(x)u^{-1}$.
\end{enumerate}
\end{lem}
\begin{proof}
Part (1) follows by performing a Tietze transformation (introducing $s$, then eliminating $t$) on the presentation of $G_\phi$ from Convention~\ref{conv:G}.

For part (2), note that for $m\ge 1$ we have $\langle t^m, F\rangle=\ker \theta$ where $\theta: G_\phi\to \mathbb Z_m$ is the epimorphism defined as $\theta|_F=0$ and $\theta(t)=[1]_m\in \mathbb Z_m$. Hence $\langle t^m, F\rangle$ is a subgroup of index $m$ in $G_\phi$. Moreover, $F\triangleleft G_\phi$ and $t^mxt^{-m}=\phi^m(x)$ for all $x\in F$, which implies that inside $G_\phi=F\rtimes_\phi \langle t\rangle$ we have $\langle t^m, F\rangle=F\rtimes_{\phi^m} \langle t^m\rangle$, and part (2) follows.

Part (3) now follows from part (2) by a Tietze transformation.
\end{proof}

\begin{lem}\label{lem:sub}
Let $H=\langle ut^m, V\rangle\le G_\phi$ be a sub-mapping torus subgroup of $G_\phi$, where $u\in F$, $m\ge 1$ and $V\le F$ is a nontrivial finitely generated subgroup such that $\phi(V)\le V$. Then with $s=ut^m$ the subgroup $H$ has the presentation
\[
H=\langle s, V| sys^{-1}=\psi(y), y\in V\rangle
\]
where $\psi:V\to V$ is an injective endomorphism defined as $\psi(x)=u\phi^m(x)u^{-1}$.
\end{lem}
\begin{proof}
Consider a group $K$ given by the presentation $K=\langle s, V| sys^{-1}=\psi(y), y\in V\rangle$. Then there is a natural homomorphism $\alpha: K\to G_\phi$, $\alpha(s)=s$, $\alpha(y)=y$ for $y\in V$ with $\alpha(K)=H$. We claim that $\alpha$ is injective.

Indeed, it is not hard to see that every element of $K$ has the form $s^{-p}ys^q$ for some $p,q\ge 0$ and $y\in V$. Hence every nontrivial element of $K$ is conjugate in $K$ to either some $1\ne y\in V$ or to $ys^n$ with $y\in V$ and $n\ne 0$. If $1\ne y\in V$ then $\alpha(y)=y\ne 1$ in $G_\phi$. Similarly, if $y\in V$ and $n\ne 0$ then $\alpha(ys^n)=ys^n\ne 1$ in $G_\phi$. Therefore $\alpha$ is injective, as claimed.

Hence $\alpha:K\to H$ is an isomorphism and the statement of the lemma follows.
\end{proof}

\begin{prop}\label{prop:sub-fbc}
Let $H=\langle s, V\rangle\le G_\phi$ be a sub-mapping torus subgroup of $G_\phi$, where $u\in F$, $m\ge 1$, $s=ut^m$, and $V\le F$ is a nontrivial finitely generated subgroup such that $\phi(V)\le V$. Let $\psi:V\to V$ be an injective endomorphism of $V$ defined as in Lemma~\ref{lem:sub}, that is $\psi(x)=u\phi^m(x)u^{-1}$ for $x\in V$.

Suppose that $\phi:F\to F$ is an automorphism of $F$. 

Then $sVs^{-1}=V$, $\psi$ is an automorphism of $V$ and $H=V\rtimes_\psi \langle s\rangle$ is a free-by-cyclic group.

\end{prop}
\begin{proof}
Since $H$ is a sub-mapping torus of $G_\phi$, we have $sVs^{-1}\le V$. Conjugation by $s$ is an automorphism $\phi$ of $F$. Thus if $sVs^{-1}\le V$ is a proper subgroup of $V$ then
\[
V\lneq s^{-1}Vs\lneq s^{-2}Vs^2\lneq \dots \lneq s^{-n}Vs^n\lneq \dots
\]
is an infinite strictly ascending chain of subgroups of $F$ of fixed finite rank. This is impossible by the Higman-Takahasi Theorem (e.g. see \cite[Theorem~14.1]{KM02}). Thus $sVs^{-1}=V$ and the statement of the proposition follows from Lemma~\ref{lem:sub}. 
\end{proof}

\begin{reptheor}{thm:A}

Let $F$ be a free group, let $\phi:F\to F$ be an injective endomorphism of $F$ and let $G_\phi$
be the mapping torus group of $\phi$.

Then for every two-generator subgroup $H\le G_\phi$ either $H$ is free or $H$ is conjugate to a sub-mapping torus subgroup of $G_\phi$.
\end{reptheor}

\begin{proof}
Let $H=\langle x,y\rangle\le G_\phi$ be a two-generator subgroup of $G$. Recall that every element of $G_\phi$ has the form $t^{-p}zt^q$ for some $z\in F$ and $p,q\ge 0$. Let $x=t^{-p_1}z_1t^{q_1}$ and $y=t^{-p_2}z_2t^{q_2}$ as above. By inverting $x,y$ if necessary we may assume that $q_1\ge p_1\ge 0$ and $q_2\ge p_2\ge 0$. Put $n=\max\{q_1,q_2\}$ and let $x'=t^nxt^{-n}, y'=t^nyt^{-n}$.  Then $x'=v_1t^{n_1}, y'=v_2t^{n_2}$ where $v_1,v_2\in F$ and $n_1,n_2\ge 0$. More precisely, $n_i=q_i-p_i$ and $v_i=t^{n-p_i}z_it^{-(n-p_i)}$ for $i=1,2$. 
Moreover $H'=\langle x',y'\rangle =t^nHt^{-n}$ is conjugate to $H$ in $G_\phi$. If $n_1=n_2=0$ then $H'=\langle v_1,v_2\rangle\le F$. Therefore $H'$ is free and hence $H$ is free as well.

Thus we may assume that $n_1+n_2>0$. We may also assume that $n_1>0$ and $n_1\ge n_2$. We can then perform the Euclidean algorithm on the exponents $n_1, n_2$ of $t$ in the generators $x',y'$ and transform the generating set $x',y'$ of $H'$ to a generating set $H'=\langle ut^m, v\rangle$ where $u, v\in F$ and where $m\ge 1, m=gcd(n_1,n_2)$. If $v=1$ then $H'$ is infinite cyclic and therefore free, as required. Thus we may assume that $v\ne 1$.

	Put $s=ut^m$. By Lemma~\ref{lem:m}~(\ref{lem:m-3}), the subgroup $\langle s,F\rangle$ of $G_\phi$ has index $m$ in $G_\phi$ and has the presentation

\[
G_{\psi}=\langle F, s|\, s x s^{-1} =\psi(x), x\in F\rangle.
\]
where $\psi: F\to F$ is given by $\psi(x)=u\phi^m(x)u^{-1}$ for all $x\in F$.
	Now $H'=\langle s, v\rangle\le G_{\psi}$ is a standard two-generator
	subgroup of $G_{\psi}$ with $v\ne 1$. Hence, by
	Proposition~\ref{prop:standard}, either $H'$ is free of rank 2 or $H$
	is a standard sub-mapping torus of $G_{\psi}$.

If $H'$ is free of rank 2 then $H\cong H'$ is free of rank 2 as well.
If $H'$ is a standard sub-mapping torus of $G_{\psi}$ then $H'$ is a sub-mapping torus of $G_\phi$ and $H$ is conjugate to a sub-mapping torus of $G_\phi$, as required.
\end{proof}

\section{Application to free-by-cyclic groups}

When $G_\phi$ is the mapping torus group of an automorphism of a rank $r$ free group $F_r$, more structure is available. First we observe an immediate corollary on the distortion of a two-generator subgroup.

\begin{repcorol}{cor:distort}
	Suppose $G_\phi$ is a free-by-cyclic group $F_r \rtimes_\phi \Z$. Then
	a two-generator subgroup of $G_\phi$ is either free of rank 2, or
	an undistorted (f.g. free)-by-cyclic sub-mapping torus subgroup of $G_\phi$.
\end{repcorol}

\begin{proof}
	Suppose $H \leq G_\phi$ is two-generated and not isomorphic to $F_2$.
	Then by Theorem~\ref{thm:A} it is either cyclic or conjugate to a
	sub-mapping torus subgroup of $H$. In the second case, the sub-mapping
	torus subgroup $\langle V,s \rangle$ is itself a (f.g. free)-by-cyclic group: since $s$ induces an automorphism of $F_r$, in fact we have $sVs^{-1}=V$, see Proposition~\ref{prop:sub-fbc} above. 
	
	It then follows from work of Mutanguha (\cite{Mut}*{Lemma 4.1} in the sub-mapping torus case, and \cite{Mut}*{Lemma 4.2} in the cyclic case) that $H$ is undistorted in $G_\phi$.
\end{proof}

Next we also impose conditions on the automorphism $\phi$. Recall that an outer automorphism $\phi\in\Out(F_r)$ is \emph{fully irreducible} if no conjugacy class of a proper free factor is $\phi$-periodic, and \emph{atoroidal} if no conjugacy class of a nontrivial element is $\phi$-periodic. Note that these properties are invariant under taking nonzero powers. The possible sub-mapping torus groups are highly constrained by the structure of fully irreducible outer automorphisms.

\begin{lem} \label{lem:invariant-fi}
	Suppose $F_r$ is a free group of finite rank $r\ge 2$ and let $\phi\in\Out(F_r)$ be a fully irreducible outer automorphism. Let $V\le F_r$ be a nontrivial finitely generated subgroup such that $\phi(V)\subseteq V$ and such that $V\ne \langle u\rangle$ where the conjugacy class $[u]$ is $\phi$-periodic. Then $V$ is of finite index in $F_r$.
\end{lem}
\begin{proof}
	This lemma is a consequence of Bestvina, Feighn, and Handel's work on
	lamination theory for $\Out(F_r)$. For brevity we do not include the
	definition of a lamination or weak convergence.  The assumptions on $V$
	imply that there exists a nontrivial element $v\in V$ whose conjugacy
	class is not $\phi$-periodic and such that iterates $\phi^n(v)\in V$
	for all $n\ge 1$. These iterates $\phi^n(v)$ can be used to construct a
	leaf of $\Lambda_\phi^+$, satisfying the definition of $V$ carrying
	$\Lambda^+_\phi$~\cite{BFH97}*{Definition~2.2}. If a finitely generated
	subgroup $V\le F_r$ carries the stable lamination $\Lambda^+_\phi$ of
	an irreducible $\phi$, then $V$ is finite index in $F_r$, as claimed~\cite{BFH97}*{Proposition~2.4}.
\end{proof}

The following consequence of Lemma~\ref{lem:invariant-fi} is a well-known
folklore result that is commonly attributed to Bestvina, Feighn, and Handel~\cite{BFH97}.

\begin{lem}\label{lem:fi-fi}
Let $\phi\in \Aut(F_r)$ be a fully irreducible atoroidal automorphism. Let $U\le F_r$ be a subgroup of finite index and let $n\ge 1$ be such that $\phi^n(U)=U$. Put $\psi=\phi^n|_U\in \Aut(U)$. Then $\psi$ is again fully irreducible and atoroidal. 
\end{lem}
\begin{proof}
Since $\phi$ has no nontrivial periodic conjugacy classes in $F_r$, it follows that $\psi$ has no nontrivial periodic conjugacy classes in $U$, so that $\psi\in\Aut(U)$ is atoroidal. Suppose that $\psi$ is not fully irreducible. Then there exists a nontrivial proper free factor $V$ of $U$ such that $\psi^m(V)=u^{-1} Vu$ for some $u\in U$ and $m\ge 1$. Thus $u\phi^{mn}(V)u^{-1}=V$. Note that $V$ has infinite index in $U$ and hence in $F_r$ as well. The automorphism $\phi^{mn}$ of $F_r$ is fully irreducible in $\Aut(F_r)$ since $\phi$ is fully irreducible. Hence $\phi'\in \Aut(F_r)$, $\phi'(w)=u\phi^{mn}(w)u^{-1}$, where $w\in F_r$, is also fully irreducible in $\Aut(F_r)$ as it equals $\phi^{mn}$ in $\Out(F_r)$. For $\phi'$ we have $\phi'(V)=V$, where $V\le F_r$ is a nontrivial subgroup of infinite index in $F_r$. Note that $\phi'$ is also atoroidal. Then Lemma~\ref{lem:invariant-fi} implies that $V$ has finite index in $F_r$, yielding a contradiction. 
\end{proof}

\begin{reptheor}{thm:B}
Let $F_r$ be a free group of finite rank $r\ge 2$, let $\phi\in \Out(F_r)$ be a fully irreducible  outer automorphism and let $G_\phi$
be the mapping torus group of $\phi$.

Then for every two-generator subgroup $H\le G_\phi$ either $H$ is free, free abelian, a Klein bottle group, or $H$ has finite index in $G_\phi$.

\end{reptheor}
\begin{proof}
Let $H\le G_\phi$ be a two-generator subgroup. By Theorem~\ref{thm:A}, either $H$ is free or $H$ is conjugate to a subgroup $H'\le G_\phi$ such that $H'$ is a sub-mapping torus subgroup.

If $H$ is free then the conclusion of Theorem~\ref{thm:B} holds. 

Suppose now that $H$ is conjugate to a subgroup $H'\le G_\phi$ such
that $H'$ is a sub-mapping torus subgroup of $G_\phi$. Thus $H'=\langle
s, V\rangle$ where $V\le F_r$ is a nontrivial finitely generated
subgroup, $s=ut^m$ with $u\in F_r$ and $m\ge 1$ and, moreover,
	$sVs^{-1}\le V$. By Lemma~\ref{lem:m}~(\ref{lem:m-3}), the subgroup $\langle
s,F_r\rangle$ has index $m$ in $G_\phi$ and, moreover, $\langle
s,F_r\rangle=G_\psi$ where $\psi:F_r\to F_r$ is the automorphism given
by $\psi(x)=u\phi^m(x)u^{-1}$ for all $x\in F_r$. Then $\psi=\phi^m$ in
$\Out(F_r)$ and in particular $\psi$ is also fully
irreducible. We also have $\psi(V)\le V$ for a nontrivial finitely
generated subgroup $V\le F_r$. Since $\psi$ is an automorphism of $F_r$, it follows that $\psi(V)=V$.

If $V$ has rank $1$, then $V = \langle c\rangle$ is infinite cyclic and
either $\psi(c)=c$ or $\psi(c)=c^{-1}$. If $\psi(c)=c$, then $H\cong
H'=\langle s, V\rangle\cong \mathbb Z\times \mathbb Z$. If
$\psi(c)=c^{-1}$ then 
\[
H\cong H'=\langle s, V\rangle=\langle s, c| scs^{-1}=c^{-1}\rangle,
\]
the Klein bottle group.

If $V$ has rank at least $2$ then, by Lemma~\ref{lem:invariant-fi} $V$ has
finite index in $F_r$. 
Therefore $H'=\langle s, V\rangle$ has finite
index in $G_\psi$ and therefore also in $G_\phi$.
Since $H$ is conjugate to $H'$ in $G_\phi$, it follows that $H$ has
finite index in $G_\phi$, as required.
\end{proof}

For automorphisms induced by a pseudo-Anosov homeomorphism on a punctured surface, we have an alternative to Lemma~\ref{lem:invariant-fi}. The following lemma is likely well understood in low-dimensional topology, and the specific proof given below was suggested to us by Chris Leininger.  

\begin{lem}\label{lem:sub_pA}
	Suppose $F_r$ is a free group of finite rank $r\ge 2$ and let $\phi\in\Out(F_r)$ be induced by a pseudo Anosov homeomorphism on some punctured surface $\Sigma$ with $\pi_1(\Sigma) \cong F_r$. If $V\le F_r$ is a nontrivial finitely generated subgroup that is not conjugate into a subgroup generated by a boundary curve and $\phi(V)\subseteq V$ then $V$ is finite index in $F_r$.
\end{lem}

\begin{proof}
	Suppose $V$ is infinite index. Then subgroup separability of free groups (or indeed of surface groups) implies that there is a finite index subgroup of $F_r$ with $V$ as a retract. In terms of the surface, this means that there is a finite degree cover $\Sigma_0$ of $\Sigma$ with $V$ the fundamental group of a proper subsurface of $\Sigma_0$. Let $\mathcal{C}$ be the isotopy classes of curves on the boundary of this subsurface but not of $\Sigma_0$.
	
	Let $f$ be the pseudo-Anosov homeomorphism inducing $\phi$, and recall that every power of $f$ is again pseudo-Anosov, and that if a pseudo-Anosov homeomorphism lifts to a cover, it induces a pseudo-Anosov homeomorphism of the cover. In particular, since $\Sigma_0$ is a finite degree cover, some power $f^k$ lifts to a pseudo-Anosov on $\Sigma_0$, and will still preserve $V$. But then $f^k$ permutes the curves in $\mathcal{C}$, and must be reducible, a contradiction.
\end{proof}

Using this lemma, we can reproduce algebraically a special case of the Jaco--Shalen result for 3-manifold groups, needing no assumption on orientability.

\begin{repcorol}{cor:3manifold}
	Let $G$ be the fundamental group of a fibered cusped hyperbolic 3-manifold $M$. Then every two-generator subgroup $H \le G$ is either free, free abelian, a Klein bottle group, or $H$ has finite index in $G$. 
\end{repcorol}

\begin{proof}
	Since $M$ is cusped and fibered, it is a mapping torus of a punctured surface $\Sigma$, and $G$ is a free-by-cyclic group $G_\phi$, where the underlying free group $F_r$ is the fundamental group of $\Sigma$. Since $M$ is hyperbolic, the homeomorphism in question is pseudo-Anosov. Applying Theorem~\ref{thm:A} we see that $H$ is either free or conjugate to a sub-mapping torus. As in the proof of Theorem~\ref{thm:B} suppose $H$ is conjugate to a sub-mapping torus $H' = \langle s, V \rangle$, where $V \leq F_r$ is a non-trivial finitely generated subgroup, and $s=ut^m$ with $u \in F_r$, $m \geq 1$ and $sVs^{-1} \leq V$. Considering the index $m$ subgroup $G_\psi$ with $\psi(x) = u \phi^m(x) u^{-1}$, again we see $\psi(V) = V$.
	
	Now apply Lemma~\ref{lem:sub_pA} to conclude that either $V$ is rank 1,
	and contained in a boundary subgroup, or of finite index in $F_r$. In
	the rank 1 case, as before we have $V = \langle c \rangle$ and $H 
	\cong \Z^2$ if $\psi(c) = c$ or a Klein bottle group if $\psi(c) =
	c^{-1}$. In each case, $H$ is a cusp subgroup. Otherwise $V$ is of finite index in $F_r$, and $H'$ is finite index in $G_\psi$ and hence in $G_\phi = G$. Since $H$ is conjugate to $H'$, $H$ is of finite index in $G$, as required. The final part of the statement follows as in Theorem \ref{thm:B}.
\end{proof}

Note that for $F_2=F(a,b)$ and $u=[a,b]$, for every $\phi\in\Aut(F_2)$ we have $\phi^2([u])=[u]$. Thus $F_2$ has no atoroidal automorphisms.

\begin{repcorol}{cor:A}
	Let $F_r$ be a free group of finite rank $r\ge 3$, let $\phi\in\Out(F_r)$ be a fully irreducible atoroidal outer automorphism. Then every two-generator subgroup $H\le G_\phi$ of the mapping torus group is either free or has finite index in $G_\phi$. 
\end{repcorol}
\begin{proof}
	If $\phi$ is atoroidal, then so is $\psi$ in the proof above, therefore the case where $V$ is rank $1$ is excluded, and the conclusion follows. Alternately, since $\phi$ is fully irreducible and atoroidal, $G_\phi$ is hyperbolic~\cite{B00}*{Theorem~1.2}. By Theorem~\ref{thm:B} the subgroup $H$ is free, free abelian, a Klein bottle group, or has finite index. Since $G_\phi$ is hyperbolic, $H$ cannot be $\mathbb{Z}\times\mathbb{Z}$ or a Klein bottle group. The last part of the statement follows as usual.
\end{proof}

The next corollary shows that, passing to a suitable finite index subgroup of $G_\phi$, one can dispense with the possibility that a two-generator subgroup is finite index. One ingredient of the proof is the following well-known lemma, whose proof we include for completeness.
\begin{lem} \label{lem:tra} Let $G$ be a group and $G_0 \leq G$ a subgroup of finite index. Then the inclusion-induced map in homology $H_1(G_0;\Q) \to H_1(G;\Q)$ is surjective; in particular, $b_1(G_0) \geq b_1(G)$. 
\end{lem} 
\begin{proof} Denote by $\alpha \colon G_0 \to G$ the inclusion map, so that
	the inclusion-induced map in integral homology is $\alpha_{*} \colon
	H_1(G_0;\Z) \to H_1(G;\Z)$. As $[G:G_0] < \infty$, there exists a {\em
	transfer} map $\alpha^{*} \colon H_1(G;\Z) \to H_1(G_0;\Z)$ with the
	property that $\alpha_{*} \circ \alpha^{*} \colon H_1(G;\Z) \to
	H_1(G;\Z)$ is multiplication by $[G:G_0]$~\cite[Section III.9]{Br94}. Changing the coefficients to the rational numbers, this entails that $\alpha_{*} \colon H_1(G_0;\Q) \to H_1(G;\Q)$ admits a right-inverse, hence it is surjective.
\end{proof}

\begin{repcorol}{cor:G0}
For each of the following groups $G$ there exists a finite index subgroup $G_0$ such that every two-generator subgroup of $G_0$ is either free, free abelian, or the Klein bottle group.
	\begin{enumerate}
		\item\label{cor:G0-1} If $\phi\in\Out(F_r)$ is fully irreducible, $r\ge 2$ and $G = G_\phi$. Moreover, if $\phi$ is additionally atoroidal, every two-generator subgroup of $G_0$ is free.
		\item\label{cor:G0-2} If $G$ is the fundamental group of a fibered cusped hyperbolic 3--manifold.
	\end{enumerate}

\end{repcorol}

\begin{proof}
	Note that the moreover of Case~(\ref{cor:G0-1}) directly follows from the first statement, since for an atoroidal fully irreducible $\phi$ the group $G_\phi$ is word-hyperbolic and hence cannot contain $\mathbb Z\times \mathbb Z$ or Klein bottle subgroups.

In both of these cases $G=F\rtimes \mathbb Z$ where $F$ is a finitely generated free group of rank at least two. Therefore $G$ is {\em large}, namely it admits a finite index subgroup which surjects to a free noncyclic group~\cites{Bu08,Bu13,HW15}. It follows that $G$ contains a subgroup $G_0$ of finite index such that $b_1(G_0) > 2$. 

	Now let $H\le G_0$ be a two-generator subgroup. Since $H\le G$, applying Theorem~\ref{thm:B} to $G$ in Case~(\ref{cor:G0-1}) and Corollary~\ref{cor:3manifold} to $G$ in Case~(\ref{cor:G0-2}) we conclude that $H$ is either free, free abelian, the Klein bottle group or $H$ has finite index in $G$ and hence in $G_0$ as well.

However, the latter option is actually not possible. Indeed, if $H$ were finite index in $G_0$ then
 $2 \geq b_1(H) \geq b_1(G_0) > 2$, with the first inequality dictated by the fact that $H$ is two-generator, and the middle inequality due to the fact that the first Betti number is non-decreasing over finite index subgroups, by Lemma \ref{lem:tra}. That yields a contradiction, and hence $H$ is either free, free abelian, or the Klein bottle group, as required.
\end{proof}

\section{Applications to one-relator groups with torsion and RFRS groups}

Recall that for an infinite subgroup $H$ of a group $G$ the \emph{height} of $H$ in $G$~\cite{GMRS} is the supremum of all $n\ge 1$ such that there exist $g_1,\dots, g_n\in G$ with $Hg_1,\dots, Hg_n$ pairwise distinct and such that the intersection 
\[
\bigcap_{i=1}^n g_i^{-1}Hg_i
\]
is infinite.

\begin{reptheor}{thm:C}
Let $G=\langle x_1,\dots, x_r| w^n=1\rangle$ where $r\ge 2$, $w\in F(x_1,\dots, x_r)$ is a nontrivial freely and cyclically reduced word and $n$ is an integer such that $n\ge |w| \ge 2$. Then $G$ contains a subgroup of finite index $G_0$ such that every two-generator subgroup of $G_0$ is free.
\end{reptheor}

\begin{proof}
A recent result of Kielak and Linton~\cite{KL24} shows that $G$ has a subgroup of finite index $G_0$ such that $G_0\cong F\rtimes_\phi \mathbb Z$, where $F$ is a free group (necessarily of infinite rank) and where $\phi\in \Aut(F)$ is some automorphism of $F$. Let $H$ be a two-generator subgroup of $G_0=G_\phi$.
By Theorem~\ref{thm:A}, either $H$ is free or $H$ is conjugate in $G_\phi$ to a subgroup $H'\le G_\phi$ such that $H'$ is a sub-mapping torus of $G_\phi$.

If $H$ is free, the conclusion of Theorem~\ref{thm:C} holds, as required.

Suppose now that $H$ is conjugate in $G_\phi$ to a subgroup $H'\le
G_\phi$ such that $H'$ is a sub-mapping torus of $G_\phi$. Then
$H'=\langle s, V\rangle$, where $s=ut^m$, $u\in F$, $m\ge 1$ and where
$V\le F$ is a nontrivial finitely generated subgroup such that
$sVs^{-1}\le V$ and so $V\le s^{-1}Vs$. Hence $V\le s^{-(p-1)}Vs^{p-1} \le s^{-p}Vs^p$ for all $p\ge 1$. For the homomorphism $\theta:G_\phi \to \mathbb Z$, $\theta|_F=0$, $\theta(t)=1$ we have $\theta(V)=0$, $\theta(s)=m$ and $\theta(Vs^{p})=pm$ where $p\ge 0$ is arbitrary. Therefore the cosets
$\{Vs^{p}|p=0,1,2,3,\dots\}$ are pairwise distinct for every $p\ge 0$, while the
intersection $\cap_{i=0}^p s^{-i}Vs^{i} =V$ is infinite. Hence $V\le G$
has infinite height in $G$.  However, $G$ is a one-relator group with
torsion and thus is word-hyperbolic. Moreover, the assumption $n\ge |w|
\ge 2$ implies that $G$ is locally
quasiconvex~\cite{HW01}*{Theorem~1.2}. In particular, $V\le G$ is an
infinite quasiconvex subgroup and hence has finite height in
$G$~\cite{GMRS}, yielding a contradiction.
\end{proof}

We also obtain a corollary about certain RFRS groups. These are groups where
there is a residual chain such that each $G_{i+1}$ is contained in the kernel
of the map $G_i \to G_i^{\mathrm{fab}}$, the maximal torsion free quotient of
the abelianization of $G_i$~\cite{Ago08}*{Definition 2.1}.

\begin{repcorol}{cor:RFRS}
	Suppose $G$ is a finitely generated RFRS group with $b_2^{(2)}(G)=0$ and $\mathrm{cd}_\mathbb{Q}(G)=2$. Then there is a finite index subgroup $G_0$ of $G$ such that every two generator subgroup $H$ of $G_0$ is free or the mapping torus of an endomorphism of a finitely generated free group.
\end{repcorol}

\begin{proof}
	By a theorem of Fisher~\cite{Fis24}~{Theorem A}, such a group is virtually free-by-cyclic (in the more general sense, where the free kernel need not be finitely generated) and the group $G_0$ is precisely this finite index subgroup. This is exactly the situation of our Theorem~\ref{thm:A}, and so any two generator subgroup of $G_0$ is free or conjugate to a sub-mapping torus of $G_0$. The description of sub-mapping tori implies that any such subgroup is the mapping torus of an injective endomorphism of a finite rank free group (the map being given by the conjugation action of $s$ on $V$), and passing to a conjugate does not change this.
\end{proof}

\appendix

\section{The 3--manifold way \emph{by Peter Shalen}}\label{Sec:appendix}

The immediate goal of this appendix is to prove a generalization of Theorem
VI.4.1 of \cite{JS79} 
(or of the similar result proved by Tucker in \cite{Tuc77})
which does not require the $3$-manifold in
question to be orientable or ``sufficiently large.'' (A compact
$3$-manifold is said to be sufficiently large if it  contains a properly embedded, $2$-sided,
incompressible surface. Compact, orientable,
irreducible 3-manifolds that are sufficiently large are often called
Haken manifolds.)
This generalization
is stated below as Corollary \ref{appendix-prop}. I have chosen to
deduce it from a still more general result, Theorem \ref{k-gen} below,
which includes \cite[Corollary 7.2]{accs} and seems to be of
independent interest.

I will work in the PL category. It is understood that a
``manifold'' may have a boundary. The Euler characteristic of a
compact PL space $X$ will be denoted by $\chi(X)$. Base points will
generally be suppressed in statements whose truth is independent of
the choice of base points. 

A $3$-manifold $\wasM$ is termed {\it irreducible} if $\wasM$ is connected and
every $2$-sphere in $\wasM$ bounds a ball. One says that $\wasM$ is {\it
  $\PP^2$-irreducible} if $\wasM$ is irreducible and contains no
two-sided projective plane. These are the same definitions as the ones
given in \cite{hem}, except that the connectedness condition is not
made explicit there.

As on page 172 of \cite{hem}, 
a subgroup of $\pi_1(\wasM)$, where $\wasM$ is a connected
$3$-manifold, will be termed {\it
  peripheral} if
there is a
  connected subsurface 
$X$ of $\partial \wasM $ such that  the inclusion homomorphism
$\pi_1(X)\to\pi_1(\wasM )$ is injective, and  $P$ is conjugate to a
subgroup of the image of this inclusion homomorphism.

The {\it rank} of a group $H$ is the minimum cardinality of a
generating set for $H$.

Here are the main results:

\begin{apptheor}\label{k-gen}
Let $\wasM $ be a compact, $\PP^2$-irreducible  $3$-manifold, and let
$k\ge2$ be an integer. Suppose that 
every rank-$2$ free
abelian subgroup of $\pi_1(\wasM )$ is peripheral. Suppose also that
either 

(a) $\wasM$ is orientable and $\pi_1(\wasM)$ has no subgroup isomorphic to the
fundamental group of a closed, orientable surface $S$ such that
$1<\genus S<k$, or 

(b) $\pi_1(\wasM)$ has no subgroup isomorphic to the
fundamental group of a closed surface $S$ such that $\chi(S)$ is even and
$2-2k<\chi( S)<0$. 

Let $H$ be a subgroup
of $\pi_1(\wasM )$ that has rank at most $k$. Then either $H$ has
finite index in $\pi_1(\wasM)$, or $H$ is a free product of finitely
many subgroups, each of which is  a
 free abelian group of rank at most $2$ or a Klein bottle group.
\end{apptheor}

Theorem \ref{k-gen} will be seen to have the following corollary:

\begin{appcorol}\label{appendix-prop}
Let $\wasM $ be a compact, $\PP^2$-irreducible  $3$-manifold such that
every rank-$2$ free
abelian subgroup of $\pi_1(\wasM )$ is peripheral. Let $H$ be a subgroup
of $\pi_1(\wasM )$ that has rank at most $2$. Then $H$ is either a
free group, a free abelian group, a Klein bottle group, or a
finite-index subgroup of $\pi_1(\wasM )$.
\end{appcorol}

We emphasize that if $\wasM$ is assumed to be orientable and
sufficiently large, Theorem \ref{k-gen} and Corollary
\ref{appendix-prop} become essential equivalent, respectively, to
\cite[Corollary 7.2]{accs} and to \cite[Theorem
VI.4.1]{JS79} 
or  the main result of \cite{Tuc77}. Thus the new contribution of this
appendix is to remove the hypotheses ``sufficiently large'' and (more significantly)
``orientable'' from these results.

As was mentioned in the introduction, Corollary \ref{appendix-prop}
includes Corollary C of the present paper as a special case, since the
compact core of every finite-volume hyperbolic $3$-manifold (fibered
or not) satisfies the topological hypotheses of the corollary.

\begin{proof}[Proof that Theorem \ref{k-gen} implies Corollary
\ref{appendix-prop}]
To prove the corollary we apply the theorem with $k=2$. With this
choice of $k$, Alternative (b) of the hypothesis of the theorem is
vacuously true, since there are no even integers strictly between $-2$
and $0$. (If $N$ happens to be orientable, then Alternative (a)  is also
vacuously true; otherwise it obviously is not.) Hence, by the theorem, either $H$ has finite index in
$\pi_1(\wasM)$, or $H$ is isomorphic to a free product $X_1\star\cdots\star X_q$, where each $X_i$ is either a
free abelian group of rank at most $2$ and a Klein bottle group. In
the latter case, by
Grushko's Theorem, we have $\sum_{i=1}^q\rank X_i=\rank H\le2$. Hence
either each $X_i$ is a free abelian group of rank at most $1$, in
which case $H$ is free; or $m=1$ and $X_1$ has rank $2$, in which case
$H$ is a rank-two free abelian group or a Klein bottle group.
\end{proof}

The proof of Theorem \ref{k-gen} is a refinement of the proof of Theorem
VI.4.1 of \cite{JS79}. In order to emphasize this,  I will be citing only results that were available
when \cite{JS79} was written.

I will say that a group $G$ is {\it freely indecomposable} if $G$ is
neither trivial nor infinite cyclic, and is not isomorphic to the free
product of two non-trivial groups. It follows from Grushko's Theorem
that every non-trivial finitely generated group $H$ is isomorphic to a
free product $X_1\star\cdots\star X_q$, where $q\ge1$ is an integer,
and each $X_i$ has
rank at most $k$ and is either freely indecomposable or infinite
cyclic. If $H$ arises as a subgroup of a group $G$, and if some $X_i$
has finite index in $G$, then  $q=1$ and $H $ has finite index in
$G$. Hence
Theorem \ref{k-gen} will follow immediately from the following result:

\begin{prop}\label{behind k-gen}
Let $\wasM $ be a  $3$-manifold for which the hypotheses of Theorem
\ref{k-gen} hold. Let $H$ be a freely indecomposable subgroup
of $\pi_1(\wasM )$ that has rank at most $k$. Then either $H$ has
finite index in $\pi_1(\wasM)$, or $H$ is isomorphic
either to a
 free abelian group of rank $2$ or a Klein bottle group.
\end{prop}

The rest of this appendix is devoted to the proof of Proposition
\ref{behind k-gen}. It will be understood that we are given
a manifold $\wasM $ 
for which the hypotheses of Theorem
\ref{k-gen} hold, and a freely indecomposable subgroup   $H$
of $\pi_1(\wasM )$ that has rank at most $k$.
We shall denote by
$\twasM=\twasM_H$ the  connected covering space of $\wasM $ determined by the subgroup $H$, and by
$p:\twasM\to \wasM $ the covering projection.

Since $\wasM $ is $\PP^2$-irreducible, it follows from the Projective Plane Theorem
\cite[Theorem 4.12]{hem}
that $\pi_2(\wasM )=0$. If $\pi_1(\wasM )$ is finite, then $H$ has finite index
in $\pi_1(\wasM )$; hence we may assume that $\pi_1(\wasM )$ is infinite. This,
together with the triviality of $\pi_2(\wasM )$, shows that the universal cover of $\wasM $ is contractible, so that $\wasM $
is aspherical. Hence $\twasM$ is also aspherical.

According to  \cite[Theorem 8.6]{hem},
$\twasM$ has a {\it compact core}, i.e. a compact, connected
three-dimensional submanifold
$K\subset\inter \twasM$ such that the inclusion homomorphism
$\pi_1(K)\to\pi_1(\twasM)$ is an isomorphism. For any compact core $K$ of
$\twasM$, we have $\pi_1(K)\cong H$. Among all compact cores of
$\twasM$ we fix one, $K_0$, having the smallest possible number of
boundary components. 

\begin{lem}\label{irreducible core}
The manifold $\inter K_0$ contains no  two-sided projective plane, and every
$2$-sphere in $K_0$ is the boundary of a compact, contractible
submanifold of $K_0$. Furthermore, if $\twasM$ is irreducible then $K_0$
is $\PP^2$-irreducible.
\end{lem}

\begin{proof}
We have observed that $\twasM$ is aspherical; since it is
finite-dimensional, it follows that it has a torsion-free fundamental
group. Any two-sided projective plane in $\twasM$ would therefore lift to
a two-sided projective plane in the orientation cover $\twasM'$ of $\twasM$,
contradicting the orientability of $\twasM'$. In particular, $\inter K_0$
contains no two-sided projective plane.

To prove the remaining assertions, it suffices to show that if $S\subset\inter K_0$ is a $2$-sphere, then
$S$ is the boundary of a compact, contractible
submanifold $B$ of $K_0$; and that if $\twasM$ is irreducible then we can
take $B$ to be a $3$-ball. 

Since $\pi_2(\twasM)=0$, it follows from the proof of \cite[Theorem 2]{Mil62}
that the sphere $S\subset\inter K_0\subset\inter \wasM $ is the boundary of
a  compact
submanifold $B$ of $\twasM$ which is simply connected and therefore contractible. If $\twasM$ is irreducible we may take $B$ to
be a ball. Set $K=K_0\cup B$. The inclusion homomorphism
$\pi_1(K_0)\to\pi_1(K)$ is surjective; since the inclusion homomorphism
$\pi_1(K_0)\to\pi_1(\twasM)$ is an isomorphism, the inclusion homomorphism
$\pi_1(K)\to\pi_1(\twasM)$ must also be an isomorphism, i.e. $K$ is a
compact core of $\twasM$. Our choice of $K_0$ then guarantees that
$\partial K$
has at least as many components as $\partial K_0$. But $\partial K$ is
the union of all components of $\partial K_0$ that are not contained in
$\inter B$. Hence $\inter B$ contains no component of $\partial K_0$,
and therefore $B\subset K_0$.
\end{proof}

A connected $3$-manifold $Q$ is termed {\it boundary-irreducible}
if for every boundary component $T$
of $Q$, the
  inclusion homomorphism $\pi_1(T)\to\pi_1(Q)$ is injective;
  otherwise, $Q$ is said to be {\it boundary-reducible.}

\begin{lem}\label{yes it is}
The manifold $K_0$ is boundary-irreducible.
\end{lem}

\begin{proof}
Assume that $K_0$ is not boundary-irreducible. Then according to the Loop Theorem
\cite[Theorem 4.2]{hem},
there is a properly embedded disk $D$ in $K_0$ such that
$C\doteq\partial D\subset\partial K_0$ does not bound a disk in
$\partial K_0$. Consider the case in which $D$ separates $K_0$, and
let $A$ and $B$ denote the components of $K_0-D$. Then $H\cong\pi_1(K_0)$ is
isomorphic to the free product $\pi_1(A)\star\pi_1(B)$. Since $C$ does not bound a disk in
$\partial K_0$, each of the compact manifolds $\overline{A}$ and $\overline{B}$ has a boundary
component which is not a $2$-sphere. Hence $A$ and $B$ are not simply
connected, and $H$ is a free product of two non-trivial subgroups;
this contradicts the free indecomposability of $H$. Now consider the
case in which $D$ does not separate $K_0$. In this case, $H\cong\pi_1(K_0)$ is
isomorphic to the free product $\pi_1(K_0-D)\star\Z$. Hence $H$ is
either an infinite cyclic group or a free product of two non-trivial
subgroups. Again we have a contradiction to the free indecomposability of $H$. 
\end{proof}

\begin{lem}\label{no trichotomy}
Every component of $\partial K_0$ is a torus or a Klein bottle.
\end{lem}

\begin{proof}
It follows from Lemma \ref{irreducible core} that no component of
$\partial K_0$ is a projective plane. If some component $T$ of $\partial
K_0$ is a sphere, then by Lemma \ref{irreducible core}, $C$ is the
boundary of a compact,
contractible submanifold of $K_0$, which must be all of $K_0$; this
implies that $H$ is a trivial group, a contradiction to free indecomposability.
thus Alternative (i) of the lemma holds. Thus every component
of $\partial K_0$ has non-positive Euler characteristic. 

We may
assume that 
$K_0$ is not closed, as otherwise the conclusion of the lemma is
vacuously true. Hence if $h_i$
denotes the dimension of $H_i(K_0;\Q)$ for $i=0,1,2$, then
$\chi(K_0)=h_0-h_1+h_2=1-h_1+h_2\ge1-h_1$. Since $\pi_1(K_0)\cong H$
has rank at most $k$, we have $h_1\le k$ and hence
$\chi(K_0)\ge1-k$.

We claim that the latter inequality is strict. Suppose to the contrary
that $\chi(K_0)=1-k$.
Since the compact $3$-manifold
$K_0$ is not closed,  $K_0$ admits some compact PL space $L$ of
dimension at most $2$ as a deformation retract. We may give $L$ the
structure of a finite CW complex with only one $0$-cell. Let $m$ and
$n$ denote, respectively, the number of $1$-cells and the number of
$2$-cells of this CW complex. Then $1-k=\chi(K_0)=\chi(L)=1-m+n$, so
that $m-n=k$. But $H\cong\pi_1(K_0)\cong\pi_1(L)$ admits a
presentation with $m$ generators and $n$ relations; thus the
deficiency of this presentation
is equal to $k$ (in the sense that the number of generators exceeds
the number of relations by $k$). On the other hand, the rank of $H$ is at most $k$ by
hypothesis. We now invoke a theorem due to Magnus \cite{Mag39}
which asserts that
if for some integer $k$, a given group has rank at most $k$ and admits
a presentation with deficiency $k$, then the group is free of rank
$k$. Thus in this case $H$ is free of rank $k$, a contradiction to
free indecomposability. This proves that $\chi(K_0)>1-k$. 

Assume that the conclusion of the lemma is false, and let
us denote the components of $\partial K_0$ with non-zero
Euler characteristic as $S_1,\ldots,S_r$, where $r\ge1$. Since every
component of $\partial K_0$ has non-positive Euler characteristic, we
have $\chi(S_i)<0$ for $i=1,\ldots,r$. We may take the $S_i$ to be
indexed so that $0>\chi(S_1)\ge\ldots\ge\chi(S_r)$. We have
$\sum_{i=1}^r\chi(S_i)=\chi(\partial K_0)=2\chi(K_0)>2-2k$. (The
equality $\chi(\partial K_0)=2\chi(K_0)$ holds because $K_0$ is a compact,
odd-dimensional manifold with boundary: see for example \cite[p. 379,
Exercise 7.3]{massey}.) In
particular we have $0>r\chi(S_1)>2-2k$. Note also that by Lemma
\ref{yes it is}, $\pi_1(S_1)$ is isomorphic to a subgroup of
$\pi_1(K_0)$. Since $K_0$ is a compact core of the covering space
$\twasM$ of $\wasM$, it follows that $\pi_1(\wasM)$ has a subgroup
isomorphic to $\pi_1(S_1)$.

By hypothesis, one of the alternatives (a) or (b) of Theorem
\ref{k-gen} holds. In the case where (a) holds, $\twasM$ is orientable, and
hence the compact core $K_0$ of the covering space $\twasM$ is
orientable. Since $0>r\chi(S_1)>2-2k$ and $r\ge1$, we have $1<\genus S_1<k$. As $\pi_1(\wasM)$ has a subgroup
isomorphic to $\pi_1(S_1)$, this contradicts (a).

Now suppose that Alternative (b) holds and that $r=1$. Then we have
$\chi(S_1)=2\chi(K_0)$, and hence $\chi(S_1)$ is even. Since
$0>\chi(S_1)>2-2k$, and $\pi_1(\wasM)$ has a subgroup
isomorphic to $\pi_1(S_1)$, this contradicts (b).

There remains the case in which (b) holds and $r\ge2$. Since
$0>r\chi(S_1)>2-2k$, we have $0>\chi(S_1)>1-k$. Since the closed
surface $S_1$ has negative Euler characteristic, it has a two-sheeted
covering space $\tS_1$, and $\pi_1(\wasM)$ has a subgroup
isomorphic to $\pi_1(\tS_1)$. But since $\chi(\tS_1)=2\chi(S_1)$, the
integer $\chi(\tS_1)$ is even and satisfies  $0>\chi(\tS_1)>2-2k$;
again we have a contradiction to (b).
\end{proof}

\begin{proof}[Proof of Proposition \ref{behind k-gen}]

Let $\wasM $ be a  $3$-manifold for which the hypotheses of Theorem
\ref{k-gen} hold. Let $H$ be a freely indecomposable subgroup
of $\pi_1(\wasM )$ that has rank at most $k$. Then $\twasM=\twasM_H$ and $K_0$
may be defined as above, and Lemmas \ref{irreducible core},
\ref{yes it is} and \ref{no trichotomy} may be applied.

If $K_0$ is closed, then since $K_0\subset\twasM$ and the
$3$-manifold $\twasM$ is connected, we must have $K_0=\twasM$; thus $\twasM$ is
closed, and $p:\twasM\to \wasM $ is a finite-sheeted covering. Hence $H$ has
finite index in $\pi_1(\wasM )$. For the rest of the proof I shall
assume that $K_0$ has non-empty boundary.

By Lemmas \ref{yes it is} and \ref{no trichotomy}, $K_0$ is
boundary-irreducible, and each component of $\partial K_0$ is a torus
or a Klein bottle.

We define a covering space $\wasN $ of $K_0$ as follows. If every component of $\partial K_0$ is a
torus (in which case $K_0$ may or may not be orientable) we take $\wasN =K_0$ to be the trivial covering. If some
component of $\partial K_0$ is a Klein bottle, so that in particular
$K_0$ is non-orientable, we define $\wasN $ to be the orientation
covering of $K_0$. Note that in any event $\wasN $ is connected, that
each component of $\partial \wasN $ is a torus, and that the covering
$\wasN $ of $K_0$ has degree at most  $2$. We let $p':\wasN \to K_0$
denote the covering projection. Letting $\iota:K_0\to\twasM$ denote the
inclusion map, we set $\ell = p \circ\iota\circ p':\wasN \to \wasM $. Since $p$
and $p'$ are covering maps, and $K_0$ is a compact core of $\twasM$, the
maps $p'$, $\iota$ and $p$ all induce injections of fundamental
groups, and hence so does $\ell $. Now since $K_0$ is
boundary-irreducible, $\wasN $ is also boundary-irreducible. Hence
for every component $T$ of $\partial \wasN $ 
  the inclusion homomorphism $\pi_1(T)\to\pi_1(M)$ is injective, and
therefore $(\ell |T)_\sharp:\pi_1(T)\to\pi_1(\wasM )$ is injective. Since $T$ is
  a torus, it follows that the
  image of $(\ell |T)_\sharp$, which is a subgroup of $\pi_1(\wasM )$ defined
  up to conjugacy, is a free abelian group of rank $2$. By the
  hypothesis of the theorem it now follows that the
  image $P$ of $(\ell |T)_\sharp$ (a subgroup of $\pi_1(\wasM )$ defined up
  to conjugacy) is peripheral; that is, there is a
  connected subsurface 
$X$ of $\partial \wasM $ such that (1) the inclusion homomorphism
$\pi_1(X)\to\pi_1(\wasM )$ is injective, and (2) $P$ is conjugate to a
subgroup of the image of this inclusion homomorphism. It follows from
(1) and (2), together with
the asphericity of $\wasM $, 
that $\ell |T$ is homotopic to a map
whose image is contained in $\partial \wasM $. 
As this holds for every
component $T$ of $\partial \wasN $, it now follows from the homotopy
extension property for compact PL pairs that $\ell $ is homotopic to a map
$f:\wasN \to \wasM $ which is {\it boundary-preserving} in the sense that
$f(\partial \wasN )\subset\partial \wasM $.

Since $\partial K_0\ne\emptyset$, we have $\partial \wasN \ne\emptyset$. Since $f$ is boundary-preserving it
then follows that $\partial \wasM \ne\emptyset$. It follows from  
Lemmas 6.7 and 6.6 of \cite{hem} that every compact, $\PP^2$-irreducible $3$-manifold with non-empty boundary is
sufficiently large, in the sense defined on page 125 of \cite{hem}:
that is, that such a manifold contains a properly embedded, $2$-sided,
incompressible surface. According to
\cite[Corollary 13.5]{hem}, every covering space of a sufficiently large,
$\PP^2$-irreducible, compact $3$-manifold is $\PP^2$-irreducible. Thus,  in the present situation,
$\twasM$ is $\PP^2$-irreducible and in particular irreducible. It then follows from Lemma 
\ref{irreducible core} that $K_0$ is $\PP^2$-irreducible. Since
$\partial K_0\ne\emptyset$, the manifold $K_0$ is also sufficiently
large. Another application of \cite[Corollary 13.5]{hem} now shows
that the covering space $\wasN $ of $K_0$ is $\PP^2$-irreducible; and since
$\partial \wasN \ne\emptyset$, the manifold $\wasN $ is sufficiently large.

Since  $f:\wasN \to \wasM $ is homotopic to $\ell$,
it induces an injection of fundamental groups. 
It follows from \cite[Theorem 13.6]{hem} that if $f:\wasN \to \wasM $ is a boundary-preserving map
between sufficiently large,
$\PP^2$-irreducible, compact $3$-manifolds, and if $\wasN $ is
boundary-irreducible and $f_\sharp:\pi_1(\wasN )\to\pi_1( \wasM )$ is injective,
then either $f$ is homotopic to a covering map or $\wasN $ is an $I$-bundle
over a closed surface.

In the present context, in the case where $f$ is homotopic to a covering map, then the image of
$f_\sharp:\pi_1(\wasN )\to\pi_1(\wasM )$ is a finite-index subgroup $H'$ of
$\pi_1(\wasM )$. If we choose basepoints in $\wasN $, $K_0\subset\twasM$ and $\wasM $
which are compatible in the sense that $p'$ and $p$ map the respective
basepoints of $\wasN $ and $K_0$ to those of $K_0$ and $\wasM $, we have $H'\le
H\le\pi_1(\wasM )$.
Hence $H$ has finite index in $\pi_1(\wasM )$ in this case.

Now consider the case in which  $\wasN $ is an $I$-bundle
over a closed surface. In the subcase where every component of $\partial K_0$ is a
torus, our construction of the covering space $\wasN $ of $K_0$ gives that
$\wasN =K_0$, so that $K_0$ is an $I$-bundle
over a closed surface; since each component of $\partial K_0$ is a
torus, the base of the $I$-bundle is either a torus or a Klein bottle,
and hence $H\cong\pi_1(K_0)$ is either a rank-$2$ free abelian group
or a Klein bottle group.

There remains the subcase in which some component $J$ of $\partial
K_0$ is a Klein bottle.
In this subcase, according to our
construction, $\wasN $ is the orientation cover of $K_0$. Hence
$\tJ\doteq(p')^{-1}(J)$ is the orientation cover of the Klein bottle
$J$. 
Let us fix
compatible basepoints in $\tJ$ and $J$, which will also serve as
basepoints for $\wasN $ and $K_0$ respectively. In terms of these
basepoints, we have a commutative diagram of groups 
$$
\begin{CD}
{\pi_1(\tJ)}  @>{\tj_\sharp}>> {\pi_1(\wasN )}\\
@VV{(p'|\tJ)_\sharp}V @VV{p'_\sharp}V\\
{\pi_1(J)} @>{j_\sharp}>> {\pi_1(K_0)}
\end{CD}
$$
where $\tj$ and $j$ denote inclusion maps. The vertical homomorphisms
in the diagram are injective because they are induced by covering
maps, and $j_\sharp$ is injective because $K_0$ is irreducible; in
view of the commutativity of the diagram, it follows that $\tj_\sharp
$ is also injective. 

Since $\tJ$ is a boundary
component of $\wasN $, which is an $I$-bundle over a closed surface, the
image of $\tj_\sharp$ has index $1$ or $2$ in $\pi_1(\wasN )$; and since
$p$ is a two-sheeted covering map between connected spaces, 
the
image of $p'_\sharp$ has index $2$ in $\pi_1(K_0)$. Hence the image
of $p'_\sharp\circ\tj_\sharp=j_\sharp\circ (p'|\tJ)_\sharp$ has index
$2$ or $4$ in $\pi_1(K_0)$. Since the
image of $(p'|\tJ)_\sharp$ has index $2$ in $\pi_1(J)$, it follows that
the
image of $j_\sharp:\pi_1(J)\to\pi_1(K_0)$, which I shall denote by
$A$, has index $1$ or $2$ in $\pi_1(K_0)$. Since $K_0$ is boundary-
irreducible, the homomorphism $j_\sharp$ is injective and hence
$A\cong\pi_1(J)$.

If  $|\pi_1(K_0):A|=1$ then $H\cong\pi_1(K_0)\cong\pi_1(J)$, i.e. $H$
is a Klein bottle group. Finally, suppose that $|\pi_1(K_0):A|=2$. Let
$K_0^*$ denote the degree-$2$ covering space of $K_0$ determined by the subgroup
$A$ of $\pi_1(K_0)$. Since $K_0$ is $\PP^2$-irreducible and
sufficiently large, $K_0^*$ is $\PP^2$-irreducible by \cite[Corollary
13.5]{hem}. The boundary of $ K_0^*$ contains a surface $J^*$
which is mapped homeomorphically onto $J$ by the covering projection
$q:K_0^*\to K_0$; furthermore, the inclusion homomorphism
$\pi_1(J^*)\to\pi_1(K_0^*)$ is an isomorphism, which by \cite[Theorem
10.2]{hem} and the irreducibility of $K_0^*$ 
implies that $K_0^*$ is a trivial $I$-bundle over a closed
surface. 
Since $K_0$ is doubly covered by a trivial $I$-bundle over a closed
surface, it follows from \cite[Theorem
10.3]{hem}
that $K_0$ is itself an $I$-bundle over a closed surface. Since the
component $J$ of $K_0$ is a Klein bottle, the base of the $I$-bundle
$K_0$ is doubly covered by a Klein bottle, and must therefore itself
be a Klein bottle; hence $H\cong\pi_1(K_0)$ is again a Klein bottle
group.
\end{proof}

\end{document}